\pdfoutput=1
\documentclass[conference]{IEEEtran}
\IEEEoverridecommandlockouts

\usepackage[T1]{fontenc}
\usepackage[utf8]{inputenc}
\usepackage[english]{babel}
\usepackage{cite}
\usepackage{amsmath,amssymb,amsfonts}
\usepackage{stmaryrd}
\usepackage{array}
\usepackage{booktabs}
\usepackage{tikz}
\usepackage{float}
\usetikzlibrary{decorations.pathreplacing, calc}
\usepackage{hyperref}

\newtheorem{theorem}{Theorem}
\newtheorem{lemma}{Lemma}
\newtheorem{remark}{Remark}
\newtheorem{definition}{Definition}

\begin{document}

\title{Shifted S-templates and improved lower bounds\\for Schur numbers}

\author{
\IEEEauthorblockN{Nils Bengone$^{a,1}$, Amine Brouk$^{a,1}$, Max Grinsztajn$^{a,1}$, Térence Helbert$^{a,1}$, Bao Lugherini$^{a,1}$\\
Arpad Rimmel$^{b,2}$, Joanna Tomasik$^{b,2}$%
\thanks{$^1$These authors contributed equally.}%
\thanks{$^2$These authors supervised the work of the main authors.}%
}
\IEEEauthorblockA{$^{a}$CentraleSupélec, Université Paris-Saclay, 91190, Gif-sur-Yvette, France\\
$^{b}$LISN, CentraleSupélec, Université Paris-Saclay, 91405, Orsay, France}
}
\maketitle

\begin{abstract}
We present an extension of the template-based approach for Schur numbers developed by Rowley \cite{rowley-s7}. This new form of template construction, which we call \textit{shifted S-templates}, was discovered during a conversation with ChatGPT 5.5 Pro, then refined, verified, and extended to multiple shifted S-templates. These new templates generalize the first ones by giving more flexibility in the coloring. Using this added flexibility with the new way to color the \textit{special label cells} of the template, we exhibit a template which yields the recurrence $S(k+2)\geq 10S(k)+2$, improving on the classical Abbott-Hanson recurrence \cite{abbot-hanson} $S(k+2)\geq 9S(k)+4$ for the same step. Combined with the known bounds $S(6)\geq 536$ and $S(11)\geq 203~828$, this implies

$$S(8)\geq 5~362,\qquad 
S(13)\geq 2~038~282,$$

\noindent improving the previously listed lower bounds $5~286$ and $2~011~290$.
\end{abstract}

\section{Introduction}

We denote by $S(k)$ the greatest $p$ such that we are able to partition $\llbracket 1,p \rrbracket$ into $k$ sum-free subsets. A sum-free subset is a subset of $\mathbb{N}$ containing no triple $x,y,z$ such that $x+y=z$. Only the first five Schur numbers are known exactly: $S(1)=1$, $S(2)=4$, $S(3)=13$, $S(4)=44$, and $S(5)=160$, the last equality due to Heule's SAT proof \cite{heule}. The main general lower-bound constructions for Schur numbers are recursive. Schur’s original argument gives \cite{schur}
$$S(k+1) \geq 3S(k)+1,$$ while Abbott and Hanson \cite{abbot-hanson} proved 
$$S(k+2) \geq 9S(k)+4.$$ Recently, Rowley \cite{rowley-s7} introduced template-based constructions which generalize this method and produce
$$S(k+3) \geq 33S(k)+6$$ which improves in particular $S(8)$ and $S(13)$. Ageron \textit{et al.} then used SAT searches for templates and found \cite{ageron}
$$S(k+4) \geq 111S(k)+43, \qquad S(k+5) \geq 380S(k)+148$$ which improve $S(9)$, $S(10)$, $S(11)$, $S(12)$, $S(14)$, and $S(15)$. 
Templates allow one to use a valid partition to extend it and construct a valid Schur partition of a larger size. The extension produces a \textit{main block} of size $p\cdot q$ with $p$ the size of the Schur partition, and $q$ the width of the template. Then a tail can be added to slightly extend the resulting partition. 

\noindent
Fig.~\ref{usual template} summarizes the template construction used by Ageron \textit{et al.}. 

\begin{figure}[htbp]
\centering

\begin{tikzpicture}
    \foreach \i/\col/\txt in {1/red/$A_1$, 2/blue/$A_2$, 3/blue/$A_2$, 
                              4/red/$A_1$, 5/gray/$S$, 6/gray/$S$,
                              7/red/$A_1$, 8/gray/$S$, 9/gray/$S$}{
        \node[draw, minimum width=0.7cm, minimum height=0.65cm, 
              fill=\col!30, font=\scriptsize] at (\i*0.7, 0) {\txt};
    }
    \node[font=\small] at (5*0.7, -1.2) {S-template (width $q=9$)};
\end{tikzpicture}

\vspace{0.5cm}

\begin{tikzpicture}
    \foreach \i/\col/\txt in {1/orange/$1$, 2/purple/$2$, 3/purple/$3$, 4/orange/$4$}{
        \node[draw, minimum width=0.7cm, minimum height=0.65cm, 
              fill=\col!30, font=\scriptsize] at (0, -\i*0.7) {\txt};
    }
    \node[font=\small] at (1.5, -3.5) {Base valid partition $S(2)=4$};
\end{tikzpicture}

\vspace{0.5cm}

\begin{tikzpicture}
    
    \foreach \col/\txt/\couleur in {1/1/red, 2/2/blue, 3/3/blue, 4/4/red, 
                                    5/5/orange, 6/6/orange, 7/7/red, 8/8/orange, 9/9/orange}{
        \node[draw, minimum width=0.7cm, minimum height=0.65cm, 
              fill=\couleur!30, font=\scriptsize] at (\col*0.7, 0) {\txt};
    }
    
    \foreach \col/\txt/\couleur in {1/10/red, 2/11/blue, 3/12/blue, 4/13/red, 
                                    5/14/purple, 6/15/purple, 7/16/red, 8/17/purple, 9/18/purple}{
        \node[draw, minimum width=0.7cm, minimum height=0.65cm, 
              fill=\couleur!30, font=\scriptsize] at (\col*0.7, -0.7) {\txt};
    }
    
    \foreach \col/\txt/\couleur in {1/19/red, 2/20/blue, 3/21/blue, 4/22/red, 
                                    5/23/purple, 6/24/purple, 7/25/red, 8/26/purple, 9/27/purple}{
        \node[draw, minimum width=0.7cm, minimum height=0.65cm, 
              fill=\couleur!30, font=\scriptsize] at (\col*0.7, -1.4) {\txt};
    }
    
    \foreach \col/\txt/\couleur in {1/28/red, 2/29/blue, 3/30/blue, 4/31/red, 
                                    5/32/orange, 6/33/orange, 7/34/red, 8/35/orange, 9/36/orange}{
        \node[draw, minimum width=0.7cm, minimum height=0.65cm, 
              fill=\couleur!30, font=\scriptsize] at (\col*0.7, -2.1) {\txt};
    }
    
    \foreach \col/\txt/\couleur in {1/37/red, 2/38/blue, 3/39/blue, 4/40/red}
    {
        \node[draw, minimum width=0.7cm, minimum height=0.65cm, 
              fill=\couleur!30, font=\scriptsize] at (\col*0.7, -2.8) {\txt};
    }
    
    \draw[decorate, decoration={brace, amplitude=10pt}, line width=1.5pt]
    (7, 0.3) -- (7, -2.4)
    node[midway, right=9pt, font=\small, rotate=270, anchor=south] {Main block ($4 \times 9 = 36$)};
    \node[font=\small] at (5*0.7, -3.9) {Resulting partition};
\end{tikzpicture}
\caption{Construction of a sum-free partition from an S-template. The template of width $q=9$ is repeated on each row of the base $S(2)=4$. The special label $S$ is colored with the color of the corresponding row. The tail corresponds to the last row.}
\label{usual template}
\end{figure}

We introduce an extension of the template formalism. In a usual template construction, when on row $\alpha$ the positions carrying a special \textit{label} (grey in Fig.~\ref{usual coloring}) inherit the color $g(\alpha)$ of the base coloring \cite{ageron}. In a shifted S-template, we allow such positions to inherit instead the color of a neighbor of $\alpha$ $(g(\alpha -1),g(\alpha +1), \ldots)$ from the base partition, as illustrated in Fig.~\ref{shifted coloring}. These shifted references, called shifted special \textit{labels}, create additional flexibility. While the inequality 
$$S(k+2) \geq 9S(k)+4,$$
was the best possible for S-templates,  the template of width $10$ displayed in Table \ref{cert} gives the inequality
$$S(k+2) \geq 10S(k)+2,$$
which improves the lower bounds for $S(8)$ and $S(13)$. We further generalize the construction to \textit{multiple} shifted S-templates so that shifted special labels inherit their value not only from the adjacent rows but from more distant ones: second-nearest, third-nearest, up to the $n$-th nearest row in either direction. Under our computational and time constraints, however, this extension produced no further improvement to the known bounds.

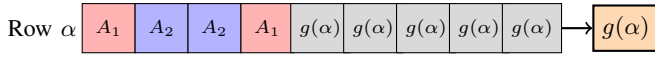
\begin{figure}[H]
\centering
\begin{tikzpicture}
    \foreach \i/\col/\txt in {1/red/$A_1$, 2/blue/$A_2$, 3/blue/$A_2$, 4/red/$A_1$,
                              5/gray/$g(\alpha)$, 6/gray/$g(\alpha)$, 7/gray/$g(\alpha)$, 8/gray/$g(\alpha)$, 9/gray/$g(\alpha)$}{
      \node[draw, minimum width=0.7cm, minimum height=0.65cm, fill=\col!30, font=\scriptsize] (c\i) at (\i*0.7, 0) {\txt};
    }
    \node[left, font=\small] at (0.4, 0) {Row $\alpha$};
    
    \draw[->, thick] (c9.east) -- ++(0.4, 0)
      node[right, draw, fill=orange!30, minimum width=0.8cm, minimum height=0.6cm, font=\small] {$g(\alpha)$};
\end{tikzpicture}
      
\caption{Usual template coloring}
\label{usual coloring}
\end{figure}

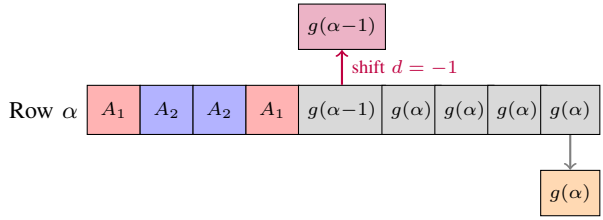
\begin{figure}[H]
\centering
\begin{tikzpicture}
    \foreach \i/\col/\txt in {1/red/$A_1$, 2/blue/$A_2$, 3/blue/$A_2$, 4/red/$A_1$}{
      \node[draw, minimum width=0.7cm, minimum height=0.65cm, fill=\col!30, font=\scriptsize] (s\i) at (\i*0.7, 0) {\txt};
    }
    \node[draw, minimum width=1.15cm, minimum height=0.65cm, fill=gray!30, font=\scriptsize] (s5) at (5*0.7 + 0.22, 0) {$g(\alpha{-}1)$};

    \foreach \i/\decal in {6/0.44, 7/0.44, 8/0.44, 9/0.44}{
      \node[draw, minimum width=0.7cm, minimum height=0.65cm, fill=gray!30, font=\scriptsize] (s\i) at (\i*0.7 + \decal, 0) {$g(\alpha)$};
    }

    \node[left, font=\small] at (0.35, 0) {Row $\alpha$};
    \node[draw, fill=purple!30, minimum width=1.15cm, minimum height=0.6cm, font=\scriptsize] (ga1) at (5*0.7 + 0.22, 1.1) {$g(\alpha{-}1)$};
    \node[draw, fill=orange!30, minimum width=0.6cm, minimum height=0.6cm, font=\scriptsize] (ga) at (9*0.7 + 0.44, -1.1) {$g(\alpha)$};
    \draw[->, thick, gray] (s9.south) -- (ga.north);
    \draw[->, thick, purple] (s5.north) -- (ga1.south) node[midway, right, font=\scriptsize] {shift $d=-1$};
\end{tikzpicture}

\caption{Shifted S-template coloring}
\label{shifted coloring}
\end{figure}

\section{General principle of the shifted S-templates}

Let $g: \llbracket 1,p \rrbracket \longrightarrow \llbracket 1,k \rrbracket$ be a valid coloring. The goal is to obtain a valid coloring with $k+r$ colors, by replacing each index $\alpha \in \llbracket 1,p \rrbracket$ by a row of width $q$. This expansion leads to a global array of size $qp$, which we refer to as the \textit{main block}. An integer of the main block is written
$$x=(\alpha -1)q+u, \qquad 1 \leq \alpha \leq p, \qquad 1 \leq u \leq q,$$
where $q$ is the width of the template. $r$ is the number of colors added when extending with this template. We denote by ordinary labels $(A_1,...,A_r)$, labels that always take the color $k+i$ when extending a partition, and by shifted special labels $P_d, d \in \Delta \subseteq \mathbb{Z}$ labels that when on row $\alpha$ take the color $g(\alpha +d)$. Thus, if $h$ is the new coloring obtained by extending a valid $k$-coloring, when on row $\alpha$, $A_i$ receives the color $$h((\alpha -1)q + u) = k + i,$$ and $P_d$ receives the color 
$$h((\alpha -1)q+u)=g(\alpha +d).$$
The correspondence between the base coloring, the row index \(\alpha\),
and the symbols of the template is summarized in
Fig.~\ref{fig:column_mechanism_large}.
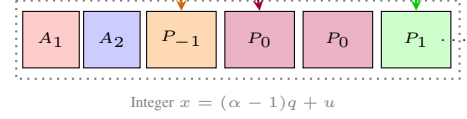
\begin{figure}[htbp]
\centering
\begin{tikzpicture}[scale=1.15, every node/.style={transform shape}]

  \node[anchor=west, font=\scriptsize\bfseries\color{black!80!black}] at (-0.2, 3.8) {1. Base Valid $k$-partition $g$:};
  
  \node[draw, thick, fill=orange!30, minimum width=1.4cm, minimum height=0.6cm, font=\tiny] (g1) at (0.9, 2.9) {$g(\alpha-1)$};
  \node[draw, thick, fill=purple!30, minimum width=1.4cm, minimum height=0.6cm, font=\tiny] (g2) at (2.6, 2.9) {$g(\alpha)$};
  \node[draw, thick, fill=green!20, minimum width=1.4cm, minimum height=0.6cm, font=\tiny] (g3) at (4.3, 2.9) {$g(\alpha+1)$};
  
  \node[above=1pt, font=\tiny\color{gray}] at (g1.north) {Row $\alpha-1$};
  \node[above=1pt, font=\tiny\color{gray}] at (g2.north) {Row $\alpha$};
  \node[above=1pt, font=\tiny\color{gray}] at (g3.north) {Row $\alpha+1$};

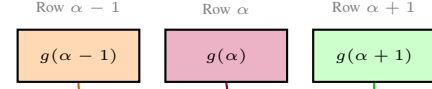
  \node[anchor=west, font=\scriptsize\bfseries\color{black!80!black}] at (-0.2, 1.8) {2. Row $\alpha$ of the template:};

  \node[draw, fill=red!20, minimum width=0.65cm, minimum height=0.65cm, font=\tiny] (a1) at (0.4, 0.9) {$A_1$};
  \node[draw, fill=blue!20, minimum width=0.65cm, minimum height=0.65cm, font=\tiny] (a2) at (1.1, 0.9) {$A_2$};
  \node[draw, fill=orange!30, minimum width=0.8cm, minimum height=0.65cm, font=\tiny] (pneg) at (1.9, 0.9) {$P_{-1}$};
  \node[draw, fill=purple!30, minimum width=0.8cm, minimum height=0.65cm, font=\tiny] (p0) at (2.8, 0.9) {$P_0$};
  \node[draw, fill=purple!30, minimum width=0.8cm, minimum height=0.65cm, font=\tiny] (p0b) at (3.7, 0.9) {$P_0$};
  \node[draw, fill=green!20, minimum width=0.8cm, minimum height=0.65cm, font=\tiny] (ppos) at (4.6, 0.9) {$P_1$};
  
  \path (ppos) -- ++(0.5,0) node[midway, font=\tiny] {$\dots$};
  \draw[thick, dotted, gray] (0.0, 1.35) rectangle (5.1, 0.45);
  
  \node[below=2pt, font=\tiny\color{gray}] at (2.5, 0.45) {Integer $x = (\alpha - 1)q + u$};

  \draw[->, >=stealth, color=purple!80!black, thick] (g2.south) -- (p0.north)
        node[pos=0.4, right, font=\tiny]{};
        
  \draw[->, >=stealth, color=orange!80!black, thick] (g1.south) .. controls (0.9, 2.0) and (1.9, 2.0) .. (pneg.north)
        node[pos=0.3, left, font=\tiny]{};
        
  \draw[->, >=stealth, color=green!80!black, thick] (g3.south) .. controls (4.3, 2.0) and (4.6, 2.0) .. (ppos.north);

\end{tikzpicture}
\caption{Mapping mechanism of a shifted S-template}
\label{fig:column_mechanism_large}
\end{figure}
\section{Formalism}

We now make the construction precise and prove that a valid template yields a valid coloring. From now on, $\Delta \subseteq \mathbb{Z}$ is the set of shifts allowed in the template. 
\subsection{Row types and licit shifts}
When choosing $\Delta=\{-1,0,1\}$, the template is only composed of three rows, $F,M,L$ (first, middle, last); authorized shifted special \textit{labels} are: 
\begin{center}
\renewcommand{\arraystretch}{1.3} 
\begin{tabular}{lll}
for $F$ : & $\{P_d : d \in \{0,1\}\}$ & (forbidden: $P_{-1}$), \\
for $M$ : & $\{P_d : d \in \Delta\}$  & (all licit), \\
for $L$ : & $\{P_d : d \in \{-1,0\}\}$ & (forbidden: $P_{1}$).
\end{tabular}
\end{center}
When choosing $\Delta = \{-n,...,n\}$, the template is composed of $2n+1$ rows, $F_1,...,F_n,M,L_1,...,L_n$. For $\alpha \in \llbracket 1,p \rrbracket$, set $\ell ( \alpha )=\alpha$ and $\rho ( \alpha)=p+1-\alpha$. Assume $p \geq 2n+1$, so that the zones below are disjoint. The type of $\alpha$ is
$$
\left\{ 
\begin{array}{lll}
     F_i, & \mbox{if} & \ell(\alpha)=i \leq n, \\
     L_j, & \mbox{if} & \rho(\alpha)=j \leq n, \\
     M, & \multicolumn{2}{l}{\mbox{otherwise.}} \\
\end{array}
\right.
$$
Authorized shifted special labels are in this case:
\begin{center}
\renewcommand{\arraystretch}{1.3} 
\setlength{\tabcolsep}{2pt}       
\begin{tabular}{rlll}
for $F_i$ : & $\{P_d : d \in \Delta, \, d \geq 1-i\}$ & (forbidden: & $P_{-i}, \dots, P_{-n}$)\\
for $M$ :   & $\{P_d : d \in \Delta\}$                & (all licit),& \\
for $L_j$ : & $\{P_d : d \in \Delta, \, d \leq j-1\}$ & (forbidden: & $P_{j}, \dots, P_{n}$).
\end{tabular}
\end{center}

With no loss of generality, we will call every type of shifted S-template (multiple or not) a shifted S-template.

\subsection{Addition inside the main block}
Let $x=(\alpha -1)q+u$ and $y=(\beta -1)q+v$ be two integers of the main block. Let: 
$$\epsilon = \begin{cases} 
0, &\text{if } u+v \leq q,\\ 
1, &\text{if } u+v > q,
\end{cases} \qquad w = u+v-\epsilon q \in \llbracket 1,q\rrbracket.$$
Thus, if $z=x+y$ is also inside the main block, it can be written $z=(\gamma -1)q+w$ with 
$$\gamma = \alpha + \beta - 1 + \epsilon,$$
where $\epsilon$ is the column carry.

\subsection{Type transitions}
To verify that a template is admissible, we must ensure that no combination of rows creates a monochromatic sum. The difficulty comes from the fact that the applied base coloring can have an arbitrary size $p$. We prove that the number of cases to verify is finite and independent of $p$.

From $\gamma = \alpha + \beta - 1 + \epsilon$, we obtain: 
\begin{equation}\label{ell-rho}
\ell(\gamma)=\ell(\alpha)+\ell(\beta)-1+\epsilon, \qquad \rho(\gamma)=\rho(\alpha)-\ell(\beta)+1-\epsilon.
\end{equation}

\begin{lemma}[Transition]\label{transition}
Assume $p \geq 3n$ and $\gamma \leq p$ (the sum stays in the block). Then:
\begin{enumerate}
\item[(a)] At most one of the two sources is of a $L$ type. If $\alpha$ is of type $L_j$, then $\beta$ is of type $F_i$ with $i\leq j-\epsilon$, and $\gamma$ is of type $L_{j-i+1-\epsilon}$.
\item[(b)] $\gamma$ is of a $F$ type $F_m$ only when both sources are of a $F$ type; then $m=\ell(\alpha)+\ell(\beta)-1+\epsilon$.
\item[(c)] Consequently, the set of realizable triples $(\tau_\alpha,\tau_\beta,\tau_\gamma)$ is finite and independent of $p$, and the licit shifts on each of the three rows depend only on its type.
\end{enumerate}    
\end{lemma}

\begin{proof}

(a) If $\alpha$ and $\beta$ were both top, then $\rho(\alpha),\rho(\beta)\leq n$, and the symmetric form of Eq.~\eqref{ell-rho}, namely $\rho(\gamma)=\rho(\alpha)+\rho(\beta)-p-\epsilon$, would give $\rho(\gamma)\leq 2n-p-\epsilon<0$, hence $\gamma>p$: excluded. So assume $\alpha$ is of type $L_j$, i.e. $\alpha=p-j+1$ with $j\leq n$. The condition $\gamma\leq p$ reads $p-j+\beta+\epsilon\leq p$, i.e. $\beta\leq j-\epsilon\leq n$; thus $\beta$ is of a bottom type $F_\beta$, and we set $i=\beta$. By Eq.~\eqref{ell-rho}, $\rho(\gamma)=j-i+1-\epsilon\in \llbracket 1,j \rrbracket \subseteq \llbracket 1,n \rrbracket $ since $1\leq i\leq j-\epsilon$; hence $\gamma$ is of type $L_{j-i+1-\epsilon}$.

(b) $\ell(\gamma)\leq n$ requires, by Eq.~\eqref{ell-rho}, that $\ell(\alpha)+\ell(\beta)\leq n+1$, hence $\ell(\alpha),\ell(\beta)\leq n$, \textit{i.e.} two bottom sources; $m=\ell(\gamma)$ is then given by Eq.~\eqref{ell-rho}.

(c) Types depend only on the distances $\ell,\rho$ bounded by $n$, and the constraints above involve only these bounded values and $\epsilon$. For $p\geq 3n$ they no longer depend on $p$. Since there are $2n+1$ types, the set of realizable triples is finite and fixed.

\end{proof}
\vspace{0.5cm}
\begin{remark}
    Admissibility is therefore checked by running over the column pairs $(u,v)\in\llbracket 1,q\rrbracket^2$ (which fixes $\epsilon$ and $w$) and the realizable type triples of Lemma~\ref{transition}: a finite enumeration. The hypothesis $p\geq 3n$ has no practical consequence, since the base colorings used satisfy $S(k)\gg n$.
\end{remark}

\subsection{Admissible templates}

\begin{definition}[Admissible template]\label{adm}
A template of width $q$, with $r$ new colors and shifts $\Delta=\{-n,...,n\}$, using the types $F_1,...,F_n,M,L_1,...,L_n$, is admissible if, for every realizable transition (Lemma~\ref{transition}) and every pair of columns $(u,v)$:
\begin{enumerate}
\item if the three cells carry the same ordinary label $A_i$, the configuration is forbidden;
\item if the three cells carry shifted special labels $P_a$, $P_b$, $P_c$, the configuration is allowed only when $c=a+b+1-\epsilon$;
\item no cell uses a shift illicit for its type.
\end{enumerate}
\end{definition}

Mixed configurations cannot be monochromatic: an ordinary label $A_i$ takes a new color in $\{k+1,...,k+r\}$, whereas a shifted special label takes a base color in $\{1,...,k\}$.

\subsection{The multiplication and tail theorems}

\begin{theorem}[Multiplication]\label{mult}
If there is an admissible template of width $q$ with $r$ new colors, then
$$S(k+r)\geq q\,S(k).$$
\end{theorem}

\begin{proof}

Let $p=S(k)$ and let $g:\llbracket 1,p\rrbracket\to\llbracket 1,k\rrbracket$ be admissible (so $p\geq 3n$). Define a coloring $h$ of $\llbracket 1,qp \rrbracket$ by applying the template to each row $\alpha$: a cell carrying $A_i$ receives $k+i$, a cell carrying $P_d$ receives $g(\alpha+d)$.

Suppose, for contradiction, $x+y=z$ with $h(x)=h(y)=h(z)$. Since base and new colors are disjoint, only two cases arise.\\
First, the three cells carry the same ordinary label $A_i$: the transition is realizable, hence forbidden by condition~1 of Definition~\ref{adm}.\\
Second, the three cells carry shifted special labels $P_a,P_b,P_c$: by condition~2, $c=a+b+1-\epsilon$. The shifts being licit (condition~3), the indices $\alpha+a,\beta+b,\gamma+c$ lie in $\llbracket 1,p\rrbracket$; with $c=a+b+1-\epsilon,$ and using $\gamma=\alpha+\beta-1+\epsilon$, we have $\gamma+c=(\alpha+\beta-1+\epsilon)+(a+b+1-\epsilon)=(\alpha+a)+(\beta+b)$, and since the three base colors coincide, $g$ would admit a monochromatic solution: contradiction.

Hence $h$ is admissible on $\llbracket 1,qp\rrbracket$, so 
$$S(k+r)\geq qp=q\,S(k).$$
\end{proof}
\vspace{0.5cm}
After the main block, we can add a tail of length $t$, at positions $qp+1,...,qp+t$, carrying fixed symbols. We then additionally check every sum that touches the tail; since the tail has fixed length, this stays a finite verification.
\vspace{0.3cm}
\begin{theorem}[Template with tail]\label{tail}
If there exists an admissible template of width $q$ with $r$ new colors and an admissible tail of length $t$, then
$$S(k+r)\geq q\,S(k)+t.$$
\end{theorem}

\begin{proof}

The construction on the main block is identical to Theorem~\ref{mult}; the $t$ terminal positions are then colored according to the tail. Sums entirely in the block are handled as before. Sums whose result falls in the tail are exactly those added to the admissibility check of the tail. Finally, a tail position can be the left term of a sum remaining in $\llbracket 1,qp+t\rrbracket$ only in cases already included in this finite verification.

\end{proof}
\section{Results}

With a SAT solver, we obtained the following results: 

\begin{table}[H]
\caption{SAT results with shifted S-templates search}
\label{results}
\centering
\renewcommand{\arraystretch}{1.2}
\begin{tabular}{c|c c c}
\toprule
$\Delta$ & $r=2, q=10$ & $r=2, q=11$ & $r = 3, q=34$ \\ 
\midrule
$\{-1,0,1\}$ & SAT & UNSAT & UNSAT \\ 
$\{-2,...,2\}$ & X & UNSAT & UNSAT \\ 
$\{-3,...,3\}$ & X & UNSAT & timeout (20h) \\ 
\bottomrule
\end{tabular}
\end{table}

For $r=2,q=11$, it remains UNSAT up to $\Delta=\{-9,...,9\}$; we did not try above that.
The multiple shifted S-templates for $r=2,q=10$ give templates that are just copies of the template with $\Delta = \{-1,0,1\}$. This template is shown in Table \ref{cert}; a tail $Q=(A,B)$ is added.

\begin{table}[h]
\caption{A shifted S-template of width $10$ with two ordinary colors $A,B$, shift set $\Delta=\{-1,0,1\}$, and tail $Q=(A,B)$.}
\label{cert}
\centering
\renewcommand{\arraystretch}{1.2}
\begin{tabular}{c|cccccccccc}
\toprule
 & 1 & 2 & 3 & 4 & 5 & 6 & 7 & 8 & 9 & 10\\
\midrule
$F$ & $B$ & $A$ & $P_0$ & $P_0$ & $B$ & $A$ & $A$ & $B$ & $P_0$ & $P_0$\\
$M$ & $P_{-1}$ & $B$ & $P_0$ & $P_0$ & $B$ & $A$ & $A$ & $B$ & $P_0$ & $P_0$\\
$L$ & $P_{-1}$ & $B$ & $P_0$ & $P_0$ & $B$ & $A$ & $A$ & $B$ & $P_0$ & $P_0$\\
\bottomrule
$Q$ & $A$ & $B$ &  &  &  &  &  &  &  & \\
\end{tabular}
\end{table}

By Theorem~\ref{tail} with $q=10$, $r=2$, $t=2$,
$$S(k+2)\geq 10\,S(k)+2.$$
This improves on the Abbott-Hanson recurrence $S(k+2)\geq 9S(k)+4$ for the same step. Combined with $S(6)\geq 536$ \cite{fredricksen-sweet} and $S(11)\geq 203~828$ \cite{ageron},
$$S(8)\geq 10\cdot 536+2=5~362$$
$$S(13)\geq 10\cdot 203~828+2=2~038~282$$
improving the previously listed lower bounds $5~286$ and $2~011~290$.

\section{Conclusion}

From an original idea of GPT 5.5 Pro, we have produced an extension of the S-templates originally designed by Rowley \cite{rowley-s7}. These shifted S-templates provide new inequalities of the form $S(k+r) \geq aS(k)+b$ and broaden the search space for lower bounds of Schur numbers. We have provided a valid shifted S-template which yields the new recurrence $S(k+2) \geq 10S(k)+2$ and thus new lower bounds for $S(8)$ and $S(13)$. For $r\geq 4$ the search stays open but is expected to be prohibitive: the relevant widths ($q\geq 112,q\geq 381$) blow up the encoding (size $\sim q^{2}\lvert\Delta\rvert^{3}$) far past the limit already reached at $r=3$.

\section{Acknowledgements}

We would like to thank Romain Ageron. During the preparation of this manuscript, it came to our attention that he had independently developed a highly similar shifted template approach. Following a fruitful exchange, we realized the close conceptual proximity of our respective unpublished works. We are deeply grateful for his scientific transparency and the constructive discussions that ensued.

\end{document}